\documentclass[12pt]{article}

\usepackage{graphicx}
\usepackage{amsmath,amssymb,amsthm}
\usepackage{graphicx,color}
\usepackage{subfigure}
\usepackage{enumerate}
\usepackage{fullpage}
\usepackage{url}
\usepackage{prettyref}
\usepackage{boxedminipage}

\def\bs{\backslash}

\newcommand{\comment}[1]{}

\definecolor{foocite}{rgb}{0,0.5,0}
\definecolor{foolink}{rgb}{0,0,1}
\definecolor{foourl}{rgb}{1,0,0}
\usepackage[colorlinks=true,citecolor=foocite,linkcolor=foolink,urlcolor=foourl]{hyperref}

\newenvironment{packed_enum}{
\begin{itemize}
  \setlength{\itemsep}{1pt}
  \setlength{\parskip}{0pt}
  \setlength{\parsep}{0pt}
}{\end{itemize}}

\newtheorem{theoremm}{Theoremm}[section]
\newtheorem{theorem}[theoremm]{Theorem}
\newtheorem{fact}[theoremm]{Fact}
\newrefformat{theorem}{Theorem~\ref{#1}}
\newrefformat{fact}{Fact~\ref{#1}}
\newrefformat{eq}{Equation~\eqref{#1}}
\newrefformat{chap}{Chapter~\ref{#1}}
\newrefformat{fig}{Figure~\ref{#1}}
\def\xthm[#1][#2][#3]{\newtheorem{#2}[theoremm]{#3} \newrefformat{#2}{#3 \ref{#11}}}
\xthm[#][definition][Definition]
\xthm[#][claim][Claim]
\xthm[#][conjecture][Conjecture]
\xthm[#][proposition][Proposition]
\xthm[#][lemma][Lemma]
\xthm[#][question][Question]
\xthm[#][corollary][Corollary]
\theoremstyle{remark}
\xthm[#][example][Example]

\begin{document}

\title{Counting large distances in convex polygons: \\a computational approach}

\author{Filip Mori\'{c}, David Pritchard\thanks{EPFL, Lausanne, Switzerland. We gratefully acknowledge support from the Swiss National
Science Foundation (Grant No. 200021-125287/1) and an NSERC Post-Doctoral Fellowship. }}

\maketitle

\begin{abstract}
In a convex $n$-gon, let $d_1 > d_2 > \dotsb$ denote the set of all distances between pairs of vertices, and let $m_i$ be the number of pairs of vertices at distance $d_i$ from one another. Erd\H{o}s, Lov\'{a}sz, and Vesztergombi conjectured that
$\sum_{i \le k} m_i \le kn$. Using a new computational approach, we prove their conjecture when $k \le 4$ and $n$ is large; we also make some progress for arbitrary $k$ by proving that $\sum_{i \le k} m_i \le (2k-1)n$. Our main approach revolves around a few known facts about distances, together with a computer program that searches all distance configurations of two disjoint convex hull intervals up to some finite size. We thereby obtain other new bounds such as $m_3 \le 3n/2$ for large $n$.
\end{abstract}




\section{Introduction}

Given a set $S$ of $n$ points in the plane, let $d_1>d_2>\dotsb$ be the set of all distances between pairs of points in $S$. It was shown by Hopf and Pannwitz in 1934~\cite{hp} that the distance $d_1$ (the diameter of $S$) can occur at most $n$ times, which is tight (e.g.~for a regular polygon of odd order). In 1987 Vesztergombi~\cite{ve2} showed that the second-largest distance, $d_2$, can occur at most $\frac{3}{2}n$ times; she subsequently~\cite{ve3} considered the version of the problem when the points are in convex position and showed that in this case the number of second-largest distances is at most $\frac{4}{3}n$. She also showed that both results are tight up to additive constants.

Let $m_i$ denote the number of times that $d_i$ occurs. It is known that $m_k \le 2kn$~\cite{ve2}, and moreover that $m_k \le kn$ for point sets in convex position~\cite{ve3}, while the following open conjecture would imply $m_k \le 2n$:
\begin{conjecture}[Erd\H{o}s, Moser~\cite{ve3,bmp}]\label{conjecture:unitdist}
The number of unit distances generated by $n$ points in convex position cannot exceed $2n$.
\end{conjecture}
\noindent A lower bound of $2n-7$ for this conjecture is known due to Edelsbrunner and Hajnal~\cite{eh91}.

For the rest of the paper we consider only point sets in convex position. One natural question is to find how large $m_{\le k} := \sum_{i \le k} m_i$, i.e.~the number of \emph{top-$k$} distances, can be in terms of $n$. The conjectured value is:
\begin{conjecture}[Erd\H{o}s, Lov\'{a}sz, Vesztergombi \cite{elv}]
\label{conjecture:con2}
The number of top-$k$ distances generated by $n$ points in convex position is at most $kn$, i.e.~$m_{\le k} \le kn$.
\end{conjecture}
\noindent Odd regular polygons prove $m_{\le k} = kn$ is possible. In \cite{elv} the bound $m_{\le k} \le 3kn$ is proven, and $m_{\le 2} \le 2n$ was shown in \cite{ve3}, verifying Conjecture~\ref{conjecture:con2} for $k=2$.

In this paper we give improved upper bounds on $m_k$ and $m_{\le k}$ for convex point sets, and more generally bounds for sums of the form $\sum_{t \in T}m_t$. Our first result is the following:
\begin{theorem}
\label{theorem:2kn}
For any $k \ge 1$, the number of top-$k$ distances generated by $n$ points in convex position is at most $(2k-1)n$, i.e.~$m_{\le k} \le (2k-1)n$.
\end{theorem}
\noindent Thus we close about half of the gap towards \prettyref{conjecture:con2}.

Next, by combining several known conditions on distances for convex point sets, and by using a computer program to carry out an exhaustive search on a finite abstract version of the problem, we prove the following.
\begin{theorem}
\label{theorem:all}
The distances generated by $n$ points in convex position satisfy the following bounds, for large enough $n$:
\begin{packed_enum}
\item $m_{\le 3} \le 3n, m_{\le 4} \le 4n;$
\item $m_3 \le \frac{3}{2}n, m_4 \le \frac{13}{8}n;$
\item $m_1+m_3 \le 2n, m_2+m_3 \le \frac{9}{4}n.$
\end{packed_enum}
\end{theorem}
\noindent In particular we verify Conjecture~\ref{conjecture:con2} for $k \le 4$ and $n$ large. For $m_3$ and $m_2+m_3$ the bound is as good as can be obtained by our abstract version of the problem, as witnessed by periodic patterns achieving $m_3 = \frac{3}{2}n$ and $m_2+m_3 = \frac{9}{4}n$, but we do not know if any convex polygon can realize these distances; we elaborate in \prettyref{sec:tightness}.

The proof of \prettyref{theorem:all} uses a computer program to make certain types of automatic deductions, as well as the following lemma to eliminate long distances ``near" the boundary:
\begin{lemma}
\label{lemma:constant-new}
For any $k\geq1$ and $\ell\geq0$, there is a constant $C(k,\ell)$ such that the following holds: in a convex polygon, if there are $\ell$ or less vertices between some vertices $a$ and $b$ such that $|ab|\geq d_k$, then the number of top-$k$ distances satisfies $m_{\le k} \le n+C(k,\ell)$.
\end{lemma}
\noindent The detailed bound we obtain is of the form $C(k, \ell) = O(k^2(k+\ell)^2)$. In an earlier version of this paper\footnote{\url{http://arxiv.org/abs/1103.0412v1}} we proved results like  ``$m_{\le 3} \le 3n + O(1)$" which are weaker for large $n$ but better for small $n$, using the following alternative lemma:
\begin{lemma}
\label{lemma:constant-old}
For any $k\geq1$ and $\ell\geq0$, there is a constant $C'(k,\ell)$ such that the following holds. In a convex polygon, at most $C'(k,\ell)$ diagonals $ab$ have both (i) $\ell$ or less vertices between $a$ and $b$ and (ii) $|ab|\geq d_k$.
\end{lemma}
\noindent In the latter, $C'(k,\ell) = O(k\ell^2)$. We do not think either lemma is tight.

In \prettyref{sec:levels} we describe \emph{levels}, a key element in our approach. In \prettyref{sec:geometry} we collect geometric facts used by the algorithm. We prove \prettyref{lemma:constant-new} in \prettyref{sec:constant}. The proof of our main result, \prettyref{theorem:all}, consists of the algorithmic approach described in \prettyref{sec:algorithm} together with our computational results stated in \prettyref{sec:results}. We conclude with suggestions for future work.

\section{Levels}\label{sec:levels}
We use the term \emph{diagonal} to mean any line segment connecting two points of $S$, including sides of the convex hull of $S$. We will partition the diagonals into $n$ \emph{levels} in the following way.
Let $S = \{a_1,a_2,\dots ,a_n\}$ be the vertex set of our convex polygon, ordered clockwise. Then \emph{level} $i$ is the set of diagonals
$$L_i := \{a_ja_k \mid j+k \equiv i \bmod{n}\},$$
where the index $i$ can be taken modulo $n$.
Equivalently, consider an auxiliary regular $n$-gon $b_1b_2\dots b_n$, then two diagonals $a_ia_j$ and $a_ka_l$ lie in the same level when the corresponding segments $b_ib_j$ and $b_kb_l$ are parallel. We illustrate this in \prettyref{fig:levels}(a).

\begin{figure}[h]
\centering
\subfigure[]{\includegraphics[scale=1.1]{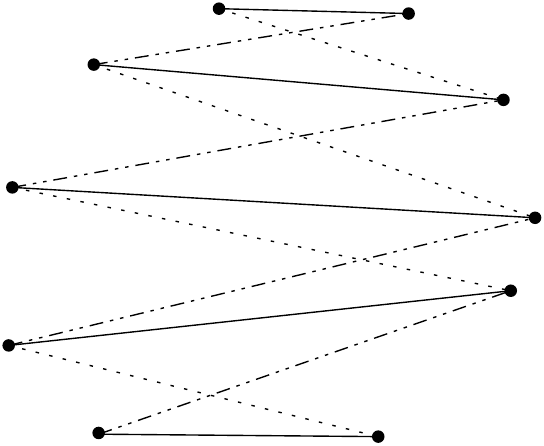}
	\label{f1}  \hspace{5mm}}
  \subfigure[]{\label{f4}
		\includegraphics[scale=1]{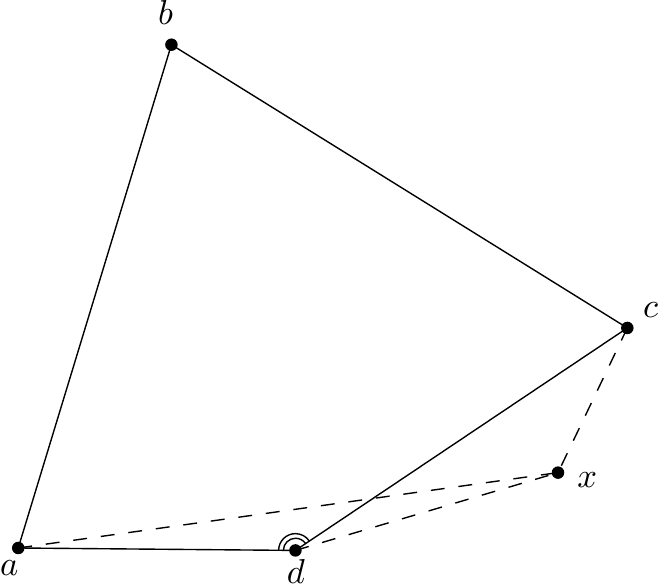}  \hspace{5mm} }
	 \caption{(a) Three consecutive levels of diagonals in a convex decagon. (b) Proof of Fact~\ref{fact:f2}.}\label{fig:levels}
\end{figure}

Levels are used in the following way to prove \prettyref{theorem:2kn}: (i.e.,~$m_{\le k} \le (2k-1)n$).
\begin{proof}[Proof of \prettyref{theorem:2kn}]
In the next section, we prove \prettyref{lemma:klevel}: in any level, there are at most $2k-1$ diagonals of length $\ge d_k$. Since there are at most $n$ levels, we are done.
\end{proof}

\section{Geometric Facts}\label{sec:geometry}
To begin this section, we collect 4 geometric facts from the literature~\cite{ve3,elv,alt}, which will be used in our computer program. For completeness, we include the proofs. The first two facts were used in~\cite{ve3,elv}.
\begin{fact}
\label{fact:f1}
If $abcd$ is a convex quadrangle, then $|ab|+|cd|<|ac|+|bd|$.
\end{fact}

\begin{proof}
Let $p$ be the intersection point of the diagonals $ac, bd$. Then by the triangle inequality,
\begin{equation}|ab|+|cd|<|ap|+|bp|+|cp|+|dp|=|ac|+|bd|\,.\tag*{\qedhere}\end{equation}\end{proof}


\begin{fact}
\label{fact:f2}
If $a,b,c,d$ are vertices of a convex polygon in clockwise order, then at least one of these four cases must occur:
\begin{packed_enum}
\item $|ax|>|ad|$ for all vertices $x$ of the polygon between $c$ and $d$, including $c$;
\item $|bx|>|bc|$ for all vertices $x$ of the polygon between $c$ and $d$, including $d$;
\item $|cx|>|bc|$ for all vertices $x$ of the polygon between $a$ and $b$, including $a$;
\item $|dx|>|ad|$ for all vertices $x$ of the polygon between $a$ and $b$, including $b$.
\end{packed_enum}
\end{fact}
\begin{proof}
Since the sum of the angles of quadrilateral $abcd$ is $2\pi$, at least one angle is non-acute. Without loss of generality
let $\angle cda\geq\frac\pi2$. Then for any vertex $x$ of the polygon between $c$ and $d$ we have that
$\angle xda\geq\angle cda\geq\frac\pi2$, and, thus, $|ax|>|ad|$ (see Figure \ref{f4}).
\end{proof}

The special case $i=j$ of the following fact appears in \cite{elv}.
\begin{fact}
\label{fact3}
If $a,b,c,d$ are vertices of a convex polygon listed in clockwise order, such that
$|bc|\geq d_i$ and $|ad|\geq d_j$, where $d_i$ and $d_j$ are the $i$-th and the $j$-th largest distances among
vertices of the polygon, then either between $a$ and $b$ or between $c$ and $d$ there are no more than $i+j-3$
other vertices of the polygon.
\end{fact}
\begin{proof}
Let us denote without loss of generality $a = a_1, b = a_x, c = a_y, d = a_z$. We will show $\min\{x-1, z-y\} \le i+j-2$ which proves the lemma. We use induction on $i+j$. The base case $i=j=1$ amounts to saying that any two non-crossing $d_1$'s must share a vertex, which follows by \prettyref{fact:f1}.

For the inductive step, we apply \prettyref{fact:f2}. Suppose that the 1st of the 4 cases happens, so $d' := a_{z-1}$ satisfies $|ad'| > |ad|$; the other cases are similar. Consequently, $|ad'| \ge d_{j-1}$. By induction, $\min\{x-1, (z-1)-y\} \le i+(j-1)-3$, from which the desired result follows.
\end{proof}
\begin{figure}[h]
\centering
\subfigure[]{\includegraphics[scale=1]{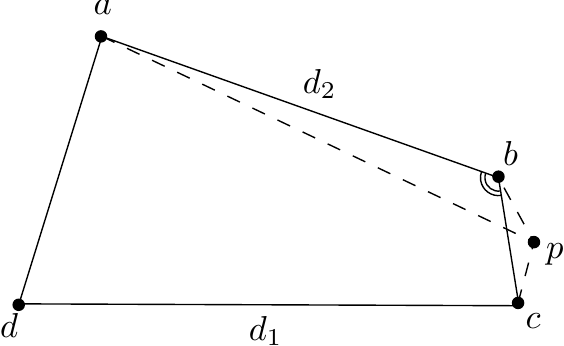}
	\label{f5}  \hspace{5mm}
	}
  \subfigure[]{\label{f6}
		\includegraphics[scale=1]{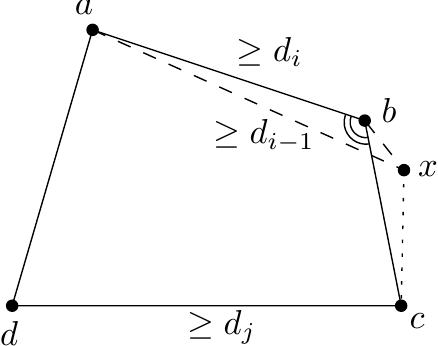}  \hspace{5mm}
      }

	 \caption{ (a) Proof of Fact \ref{fact3}, base case $i=2$, $j=1$; (b) Proof of Fact \ref{fact3}, inductive step}
\end{figure}
\comment{
Assume for simplicity that $n$ is even (the other case is almost identical).
We show it just for the level $L_1=\{a_1a_n,a_2a_{n-1},\dots,a_{n/2}a_{n/2+1}\}$, since the proof for other levels is analogous.
Let $i$ and $j$ be respectively minimum and maximum indices from $1$ to $n/2$ such that
$|a_ia_{n-i+1}|,|a_ja_{n-j+1}|\geq d_k$ (see Figure \ref{f2}).
By Fact \ref{fact3} we have that $j-i\leq2k-2$. By the choice of $i$ and $j$ all top-$k$ distances in this level
are of the form $a_{t}a_{n-t+1}$ with $i\leq t\leq j$. Therefore, there are at most $2k-1$ top-$k$ distances in $L_1$,
which finishes the proof.
}

The following is a strengthening of a result of Altman, obtained by removing all non-essential conditions from the hypothesis of~\cite[Lemma 1]{alt} but using the same proof. (He considered only the case where $|a_1a_m| = d_1$.)
\begin{fact}
\label{fact4}
Let $a_1\dots a_n$ be a convex polygon. If $1\leq i<j\leq k<\ell<m$ and $|a_1a_m|\geq\max\{|a_1a_k|,|a_ja_m|\}$, then $|a_ia_\ell|>\min\{|a_ia_k|,|a_ja_\ell|\}$.
\end{fact}
\begin{proof}
Suppose for the sake of contradiction that $|a_ia_\ell|\leq\min\{|a_ia_k|,|a_ja_\ell|\}$. Denote by $x$ and $y$ the points where $a_1a_j$ and $a_ma_k$ intersect $a_ia_\ell$ (see
Figure \ref{f7}). Repeatedly using the fact that when $s, s'$ are two sides of a triangle, $|s| > |s'|$ iff the angle opposite $s$ is larger than the angle opposite $s'$, we have
\begin{align*}\angle a_jxa_{\ell}+\angle a_kya_i&>\angle a_ja_ia_{\ell}+\angle a_ka_{\ell}a_i\geq\angle a_ia_ja_{\ell}+\angle a_{\ell}a_ka_i\\
&>\angle a_1a_ja_m+\angle a_1a_ka_m\geq\angle a_ja_1a_m+\angle
a_ka_ma_1\,.\end{align*}
However, $\angle a_jxa_{\ell}+\angle a_kya_i=\angle a_ja_1a_m+\angle a_ka_ma_1$, which gives a contradiction.
\end{proof}

\begin{figure}[h]
\centering
      \subfigure[]{\label{f7}
		\includegraphics[scale=0.9]{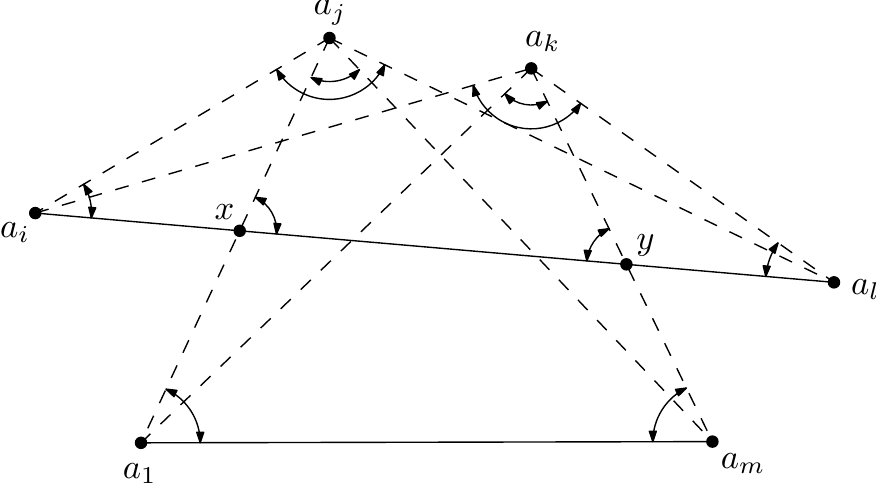}  \hspace{5mm}
      }
      \subfigure[]{\label{f8}
		\includegraphics[scale=0.9]{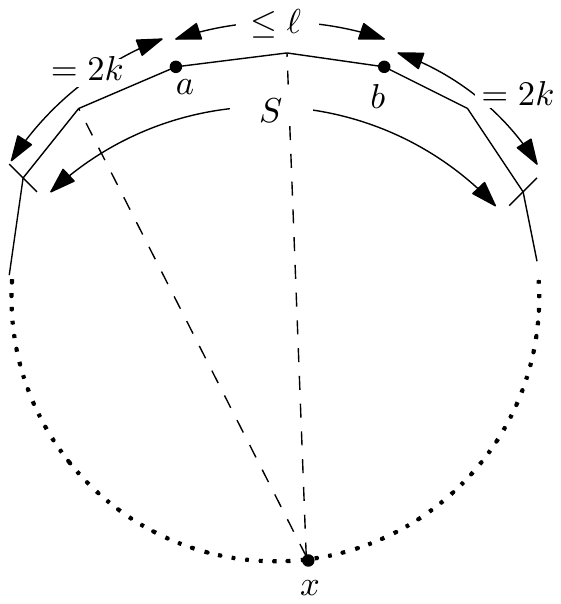}  \hspace{5mm}
      }
	 \caption{(a) Proof of Fact \ref{fact4}. (b) Proof of \prettyref{lemma:constant-new}.}\label{fig:45}
\end{figure}

\subsection{Counting Lemmas}\label{sec:constant}
First we complete the proof of \prettyref{theorem:2kn}, using Fact \ref{fact3}.
\begin{lemma}\label{lemma:klevel}
In any level there are at most $2k-1$ diagonals of length $\ge d_k$.
\end{lemma}
\begin{proof}
Without loss of generality (by relabeling), we consider the level $L_0$. The diagonals of this level are $a_ja_{-j}$, with indices modulo $n$, for $0 < j < n/2$. Let $m > 0$ (resp.~$M$) be the minimal (resp.~maximal) $j$ such that $|a_ja_{-j}| \ge d_k$. Then by Fact \ref{fact3}, we see that $M-m-1 \le k+k-3$. So the number of top-$k$ diagonals in $L_0$ is bounded by $|\{m, m+1, \dotsc, M\}| =M-m+1\le 2k-1$, which gives the corollary.
\end{proof}

Next, we give the proof of \prettyref{lemma:constant-new}, which is needed in order to argue that our computational approach is correct.
\begin{proof}
We want to show that if $|ab| \ge d_k$, and $a$ and $b$ are separated by at most $\ell$ vertices, then the number of top-$k$ distances satisfies $m_{\le k} \le n+O(k^2(k+\ell)^2)$. Let $S$ be the interval obtained from this $[a, b]$ by extending onto $2k$ further points in both directions. By Fact \ref{fact3}, all edges of length $\geq d_k$ have at least one endpoint in $S$. Note $|S| = O(k+\ell).$

We will show an upper bound of $n+O(k^2(k+\ell)^2)$ on the number of edges $sx$ of length $\ge d_k$, with $s \in S, x \in V \bs S$. This will complete the proof since the only other top-$k$ distance edges must lie with both endpoints in $S$, and there are at most $O(k+\ell)^2$ such edges.

The key observation is that in the bipartite graph between $S$ and $V \bs S$ consisting of these edges, all but a constant number of vertices in $V \bs S$ have degree 1. Specifically, if $sx, s'x$ are both edges in this graph, then the location of $x$ is uniquely determined by $s, s', |sx|,$ and $|s'x|$; it follows that $\sum_x{\tbinom{\deg(x)}{2}}$ is at most $O((k+\ell)^2k^2)$, and consequently $\sum_{x : \deg(x) > 1} \deg(x) = O((k+\ell)^2k^2)$. We are then done by counting the endpoints of degree-1 vertices, of which there are at most $n$.
\end{proof}
\section{The Algorithm}\label{sec:algorithm}
The algorithm we use to prove \prettyref{theorem:all} examines distances among finite configurations of points in the plane. Informally, we examine all possible configurations of a bounded size, where a configuration includes all occurrences of top-$k$ distances in a few consecutive levels, and we try to establish that not too many top-$k$ distances can occur per level, averaged over a small interval of levels. Thus ultimately, the argument in our proof decomposes any global point set into local configurations of bounded size.

\subsection{The Goal}
Our computational goal will be to bound the number of long distances which can occur in a consecutive sequence of several levels. We begin by re-proving (for large $n$) Vesztergombi's result on counting the second-largest distances; it illustrates the type of computational result we need.
\begin{proposition}\label{proposition:ves}
We have $m_2 \le \frac{4}{3}n$ for large enough $n$.
\end{proposition}
\begin{proof}
We prove the theorem for $n \ge 3 \cdot C(16, 2)$ with $C$ as in \prettyref{lemma:constant-new}.
Let a \emph{special diagonal} be a diagonal of length $d_2$ or longer, whose endpoints are separated by at most 16 vertices. If there is any special diagonal, we are done by \prettyref{lemma:constant-new}. So we may assume there are no special diagonals.

Using our computer program, we establish the following lemma.
\begin{lemma}\label{lemma:ves}In every point set $S$ without special diagonals, for every level $i$, at least one of the following is true:
\begin{packed_enum}
\item at most $1 = \lfloor 1 \cdot \frac{4}{3} \rfloor$ diagonal in level $i$ has length $d_2$;
\item at most $2 = \lfloor 2 \cdot \frac{4}{3} \rfloor$ diagonals in levels $i$ and $i+1$ have length $d_2$;
\item at most $4 = \lfloor 3 \cdot \frac{4}{3} \rfloor$ diagonals in levels $i, \dotsc, i+2$ have length $d_2$;
\item at most $5 = \lfloor 4 \cdot \frac{4}{3} \rfloor$ diagonals in levels $i, \dotsc, i+3$ have length $d_2$.
\end{packed_enum}
\end{lemma}
Now let us see how this gives the desired result.
Taking $i=1$, the four cases above establish that for some $1 \le \gamma_1 \le 4$, the number of $d_2$'s in levels $1, \dotsc, \gamma_1$ is at most $\frac{4}{3}\gamma_1$. Applying the same logic to $i=\gamma_1+1$, we get that there is some $1 \le \gamma_2 \le 4$ such that the number of $d_2$'s in levels $\gamma_1+1, \dotsc, \gamma_1+\gamma_2$ is at most $\frac{4}{3}\gamma_2$.

We continue defining further $\gamma_i$'s in the same way until $\sum_{i=1}^x \gamma_i \equiv \sum_{i=1}^y \gamma_i \pmod{n}$ for some $x<y$. Summing a contiguous subset of these bounds, the number of $d_2$'s in levels from $1+\sum_{i=1}^x \gamma_i$ to $\sum_{i=1}^y \gamma_i$ is at most $\frac{4}{3}$ per level on average. But this sum counts each of the $n$ levels an equal number of times, so the number of $d_2$'s overall is at most $\frac{4}{3}n$.
\end{proof}

The computer program's goal is thus to prove a general version of Lemma~\ref{lemma:ves}: given a \emph{target ratio} $\alpha$ and \emph{target distances} (a subset of $\{d_1, d_2, \dotsc, d_k\}$), find a constant $m$ so that every level $i$ admits $1 \le m' \le m$ such that $\le m' \cdot \alpha$ target lengths occur in levels $i, \dotsc, i+m'$. The program searches for a point set with $> \alpha$ target diagonals in level 1, $> 2\alpha$ in level 2, etc. If the search terminates, the above proof shows the number of target distances is $\le \alpha n$. The hypothesis that no special diagonals exist is used only indirectly by the program, explained below.

Our algorithm works with \emph{configurations} consisting of two disjoint intervals of points, and an assignment of a distance from $\{d_1, d_2, \dotsc, d_k, ``< d_k"\}$ to each diagonal spanning the two intervals. We thereby obtain analogues of Lemma~\ref{lemma:ves} by checking all possible configurations up to some finite size. For this to work, \prettyref{fact:f2} is crucial since it implies that all of the top-$k$ distances in $\ell$ consecutive levels have all of their endpoints in two intervals of bounded size. We use an incremental branch-and-bound search: it exhaustively searches all possibilities, but in an efficient way where large sections of the search space can be eliminated at once. Each individual step of the algorithm corresponds to an application of one of the Facts \ref{fact:f1}--\ref{fact4}. The lack of special diagonals allows us to focus on \emph{disjoint} interval pairs.
The Java implementation is available at
\begin{center}\url{http://sourceforge.net/projects/convexdistances/}.\end{center}

\subsection{Configurations}
In more detail, our algorithm maintains a set of \emph{configurations}. Each configuration has two disjoint intervals of points from $S$; then for each diagonal generated by one point from each interval, the configuration stores a set of possible values for the distance between those two points. Arbitrarily name one interval the \emph{top} and denote its points as $\{t_i\}_i$, with $t_{i+1}$ following $t_i$ in clockwise order, and name the other interval the \emph{bottom} with points $\{b_i\}_i$, and $b_{i-1}$ following $b_i$ in clockwise order. Then we denote the set of possible distances between $t_i$ and $b_j$ as $D[i, j]$; in each configuration $D[i, j]$ is a subset of $\{1, 2, \dotsc, k, \infty\}$ where $x \in D[i, j]$ means that $d_x$ is a possible value for the distance $|t_ib_j|$, while $\infty \in D[i, j]$ means that it is possible for $|t_ib_j|$ to be shorter than $d_k$. (So typical steps in our program use special cases to reason with ``$d_\infty$" distances correctly.) Reiterating, a configuration consists of a top interval of indices, a bottom interval of indices, and for each top-bottom pair a subset of $\{1, 2, \dotsc, k, \infty\}$.

We assume that $t_ib_j$ is in level number $j-i$ (modulo $n$), which is without loss of generality. To gain some intuition and exhibit the notation, it is helpful to look at a couple of examples. Our examples will be drawn from actual point sets and therefore each $D[i, j]$ will be just a singleton, in contrast to the larger sets $D[i, j]$ typically occurring in the algorithm. The first example, shown in \prettyref{fig:regular}, is a regular polygon of odd order. The second example, shown in \prettyref{fig:ves}, exhibits the extremal construction of Vesztergombi for second distances~\cite{ve3}.

\begin{figure}[h]
\hfill
\includegraphics{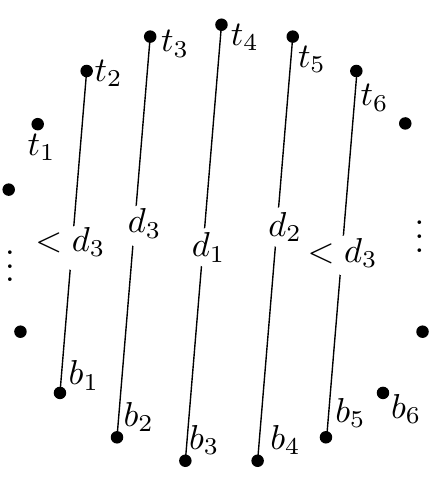}
\hfill
\includegraphics{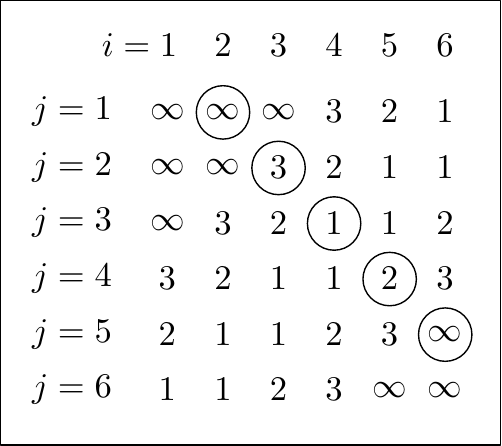}
\hfill
\phantom{}
\caption{Left: an odd regular polygon, with a top and bottom interval. Right:
the corresponding values of $D$, where entry $x$ in column $i$, row $j$ indicates $D[i, j] = \{x\}$. One level is illustrated on the left and circled on the right.}\label{fig:regular}
\end{figure}

\begin{figure}
\hfill
		\includegraphics[scale=0.9]{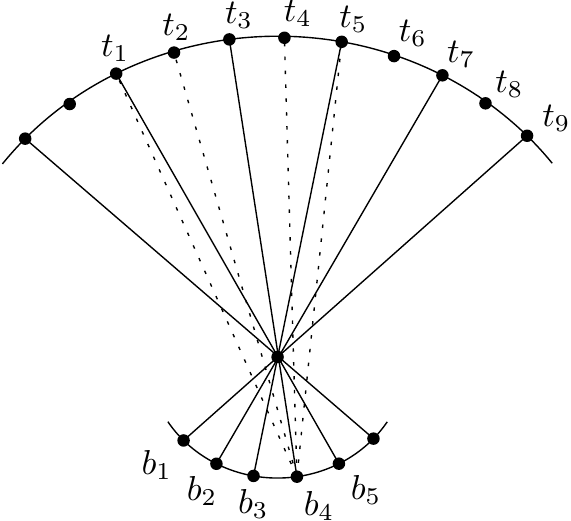}
\hfill
		\includegraphics[scale=1]{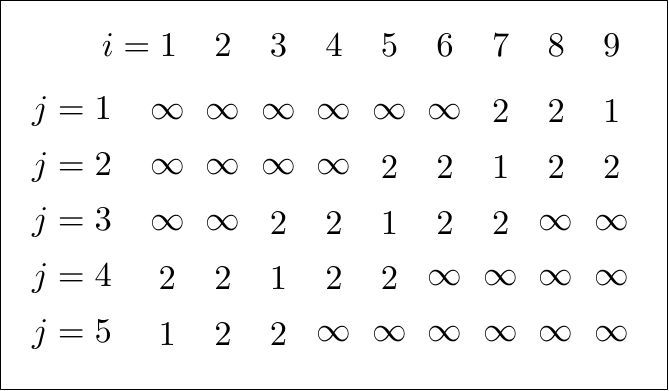}
\hfill
\phantom{}
\caption{Left: an illustration of Vesztergombi's construction with $m_2 = \frac{4}{3}n-O(1)$. Some diagonals of length $d_1$ and $d_2$ are shown (solid and dotted, respectively). Right: the corresponding configuration; again, entry $x$ in column $i$, row $j$ indicates $D[i, j] = \{x\}$.}\label{fig:ves}
\end{figure}

\subsection{Methodology}
Here is an example of a typical step in the algorithm, shown in \prettyref{fig:step}. Suppose some configuration includes points $t_1, t_2, b_2, b_1$, suppose that $D[1,1]=D[2,2]=\{2\}$, $D[1,2]=\{2,3,\infty\}$ and that $D[2,1]=\{1,2,3,\infty\}$. Then using Fact~\ref{fact:f1}, we know that $|t_1b_2|+|t_2b_1| > |t_1b_1|+|t_2b_2|$. As the right-hand side equals $2d_2$ and the maximum possible length of $t_1b_2$ is $d_2$, we can deduce that $|t_2b_1| > d_2$ and so we may update the configuration via $D[2,1] := \{x \in D[2, 1] \mid x < 2\} = \{1\}$.

\begin{figure}
\hfill
\includegraphics{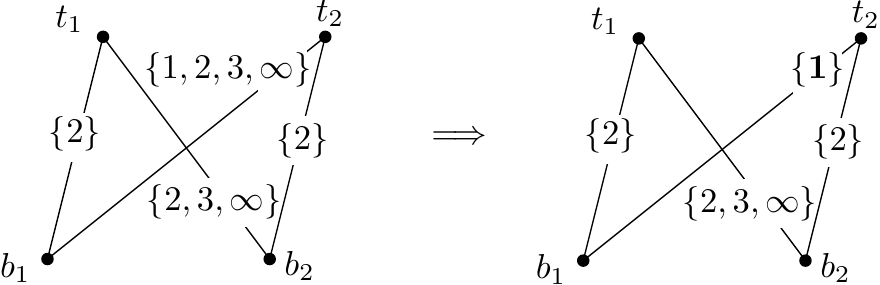}
\hfill
\phantom{}
\caption{A typical step of the algorithm, using Fact~\ref{fact:f1}.}\label{fig:step}
\end{figure}

The program uses Facts \ref{fact:f1}--\ref{fact4} in ways analogous to the above example. Whenever one of the facts is applicable, we use it to reduce the size of one set $D$ in the configuration. We use Fact~\ref{fact4} only when $a_1, a_i, a_j$ lie in the top interval and $a_k, a_l, a_n$ lie in the bottom or vice-versa.

Our algorithm also makes use of another easy observation. In any instance $S$, it cannot be true that both $d_1 + d_3 > d_2 + d_2$ and $d_1 + d_3 < d_2 + d_2$. Hence using Fact~\ref{fact:f1}, a quadruple $t, t', b', b$ (in that cyclic order) with $|tb|=|t'b'|=d_2, |tb'|=d_1, |t'b|=d_3$ cannot co-exist with another quadruple $\hat{t}, \hat{t}', \hat{b}', \hat{b}$ with $|\hat{t}\hat{b}|=d_1, |\hat{t}'\hat{b}'|=d_3, |\hat{t}\hat{b}'|=|\hat{t}'\hat{b}|=d_2$. More generally, given a configuration we can deduce from any $i, j, i', j'$ with each $D[i, j], D[i, j'], D[i', j], D[i', j']$ singletons other than $\{\infty\}$ that an inequality of the form $d_w + d_x > d_y + d_z$ is true; in testing a configuration for validity our program will reject any configuration where a contradiction arises from the set of all such pairwise inequalities. This is done by testing the associated digraph of $\tbinom{k+1}{2}$ pairs for acyclicity. (We also include arcs of the form $d_x + d_y > d_x + d_z$ whenever $y < z$.)

In some situations none of these facts are applicable; say for example, if each $D[i, j]$ is equal to $\{1, 2, \infty\}$, we cannot conclude any further information. In this case we use an approach which is similar to recursion or \emph{branch-and-bound} in this situation, which works as follows. Find some $i,j$ with $|D[i, j]|>1$, let $X$ denote $D[i,j]$. We then replace this configuration with two new configurations: each of the new ones is almost identical to the original, except that in one we take $D[i, j] = \min_{x \in X} x$ and in the other we take $D[i, j] = X \bs \{\min_{x \in X} x\}$. In a little more detail, while we are examining the levels from 1 to $L$, we only perform branching on diagonals in levels 1 to $L$, (i.e.~only when $1 \le j-i \le L$) and any other non-singleton $D[i, j]$ does not entail branching. This was faster in practice than branching on every $D[i, j]$.

\subsection{Initializing and Growing Configurations}
Recall that our theorems are all of the following form, for a set $T$ of positive integers and some real $\alpha$: \begin{equation}
\sum_{t \in T} m_t \le \alpha  n +O(1).\tag{$\spadesuit$}\label{eq:prototype}
\end{equation}
We call a \emph{target distance} any distance $d_t$ with $t \in T$. We use $k$ to represent the largest number in $T$.

We begin this detailed section by explaining why it suffices to examine configurations of bounded size to bound the number of target distances in $L$ consecutive levels. The key tool is Fact~\ref{fact3}. Namely, suppose $t_0b_1$ is any diagonal in level 1 with length $|t_0b_1| \ge d_k$, and consider any top-$k$ distance diagonal $e$ in levels $1, \dotsc, L$. If $e$ crosses $t_0b_1$, then $t_0$ (resp.~$b_1$) is within $L$ steps along the boundary from an endpoint of $e$ (resp.~the other endpoint of $e$). If $e$ and $t_0b_1$ don't cross, one endpoint of $e$ is at most $2k$ steps from $t_0$ or $b_1$ by Fact~\ref{fact3}, and the other endpoint of $e$ is at most $2k+L$ points away from the other of $t_0$ or $b_1$. Summarizing, in either case, $e$ has one endpoint in the interval $I_t$ consisting of vertices at most $2k+L$ steps from $t_0$, and $e$'s other endpoint lies in the interval $I_b$ consisting of vertices at most $2k+L$ steps from $b_1$; and this holds for all top-$k$ distance diagonals $e$ in levels $1, \dotsc, L$.

Our program makes valid deductions whenever these intervals are disjoint, which is false only when $t_0$ and $b_1$ are within $2(2k+L)$ steps of one another on the boundary. Set $\ell = 2(2k + L)$ and define a \emph{special diagonal} to be one with length $\ge d_k$ and at most $\ell$ vertices between its endpoints. Recall that $|t_0b_1| \ge d_k$, so the program's deductions are valid unless there was a special diagonal. This explains the choice of $16 = 2(2\cdot2 + 4)$ in \prettyref{proposition:ves} and justifies our general approach.

In the rest of this section we explain some of the implementation details. The program begins working with a configuration consisting of a single diagonal $t_0b_1$ of length $\ge d_k$, and we assume without loss of generality that there are no diagonals $t_ib_{i+1}$ such that $i<0$ and $|t_ib_{i+1}| \ge d_k$. Thus the top and bottom intervals begins as the singleton sets $\{t_0\}, \{b_1\}$.

We will now enlarge these configurations. Reviewing our proof strategy, the program must enumerate all possible configurations such that level 1 has more than $\alpha$ diagonals of a target length, \emph{and} levels 1 and 2 together have more than $2\alpha$, etc, with the hope being that once the number of levels is high enough we find that no such configurations exist, since this would give a result like \prettyref{lemma:ves}.

Note that, by our choice of $t_0$ and $b_1$ which normalize our indices, in any convex point set, all level-1 diagonals of the target distances are of the form $t_ib_{i+1}$ for $i>1$, and by Fact~\ref{fact3} they also satisfy $i \le 2k-2$, so crucially, their possible positions are confined to an interval of bounded size. We now determine which of these diagonals have target lengths by \emph{exhaustive guessing}, a term which simply means trying all possibilities. In detail, first, exhaustively guess the smallest $i>0$ for which $t_ib_{i+1}$ is a target distance, then the second-smallest, etc. When the top and bottom intervals are enlarged, each new $D[i, j]$ is set to $\{1, \dotsc, k, \infty\}$ by default, meaning that no assumptions are made on the distance. When $i$ is guessed as a minimal new level-1 diagonal for which $t_ib_{i+1}$ is a target distance, rather than the defaults we set $D[i, i+1] = T$ and $D[i', i'+1] := \{1, \dotsc, k, \infty\} \bs T$ for all new $i' < i$.

\begin{figure}[bht]\noindent
\begin{boxedminipage}{6.5in}
\begin{packed_enum}
\item Initialize a configuration with intervals $\{t_0\}, \{b_1\}$ and $D[0, 1]$ set to $T$ (all target distances)
\item For $L=1, 2, \dotsc$
\item[] \begin{packed_enum}
\item Extend the configurations by exhaustively guessing all diagonals of target lengths in level $L$, extending leftwards first if $L>1$, and then rightwards in all cases.
\item Keep only configurations with more than $\alpha L$ target distances in levels $1, \dotsc, L$.
\item Stop if no configurations remain. \end{packed_enum}
\medskip
\item Upon extending a configuration, {\bf check} it: \item[] \begin{packed_enum}
\item Use Facts \ref{fact:f1}--\ref{fact4} to perform deductions.
\item Check that distance pairs are consistent.
\item If $|D[i, j]| > 1$ for some diagonal $t_ib_j$ in one of the first $L$ levels, partition it into two configurations and {\bf check} both (recursively).\end{packed_enum}
\end{packed_enum}
\end{boxedminipage}
\caption{Sketch of the algorithm.}\label{fig:psuedocode}
\end{figure}

After each new diagonal is added, we re-apply Facts \ref{fact:f1}--\ref{fact4} in order to make additional deductions and eliminate any impossible configuration; and we split any non-singleton sets $D$ in the first level, as described earlier.

After this exhaustive guessing, we have collected all possible configurations. We keep only those for which level 1 has more than $\alpha$ diagonals of the target lengths. If any exist, we grow them in all possible ways to 2-level configurations, using exhaustive guessing like that explained above, except that we expand ``to the left" before expanding ``to the right" (for level 1, only rightwards expansion was needed due to our choice of $t_0$ and $b_1$). Again, we prune those which have no more than $2\alpha$ target distance in the first two levels.

We repeat the process described in the previous paragraph over and over, increasing the number of levels by 1 each time. If the program terminates eventually, it implies a result of the form like \prettyref{lemma:ves} and consequently that \eqref{eq:prototype} holds for this choice of $T$ and $\alpha$. We give a high-level review of the algorithm in \prettyref{fig:psuedocode}.

\section{Results: Proof of \prettyref{theorem:all}}\label{sec:results}
Each row in Table~\ref{table:1} corresponds to an execution of our program which terminated. In other words, each execution establishes that an analogue of \prettyref{lemma:ves} holds, and we consequently deduce \prettyref{theorem:all} using reasoning as in the proof of \prettyref{proposition:ves}. Each line proves
\begin{equation}\sum_{t \in T}m_t \le \alpha n  \textrm{ for $n > C(k, 2(2k+L))/(\alpha-1)$},\label{eq:again}\tag{$\clubsuit$}\end{equation}
where $k$ is the largest element of $T$, and $C$ is the constant from \prettyref{lemma:constant-new}. Note that the first two lines of Table~\ref{table:1} correspond to results that were already known. The running times are from a computer with a 2 GHz processor.
The program was written in Java, and is available on SourceForge\footnote{\url{http://sourceforge.net/projects/convexdistances/}}.
For $T = \{1, 2, 3, 4, 5\}$ or $T = \{5\}$ the program ran out of memory before obtaining any reasonable result.
\begin{table}[htb]
\begin{center}\begin{tabular}{ccccc}
$T $&$ \alpha $&$ L $&time (s) & tightness of result \\
\hline
$\{1,2\} $&$ 2 $&$2 $&$<1 $ & tight (odd regular) \\
$\{2\} $&$ 4/3 $&$4 $&$<1 $ & tight \cite{ve3} \\
$\{1,2,3\} $&$ 3 $&$ 3 $&$<1 $ & tight (odd regular) \\
$\{3\} $&$ 3/2 $&$ 9 $&$ 5 $ & abstractly tight, Fig.~\ref{fig:m3} \\
$\{2,3\} $&$ 9/4 $&$ 6 $&$ 1 $ & abstractly tight, Fig.~\ref{fig:m23} \\
$\{1,3\} $&$ 2 $&$ 4 $&$ <1 $ & tight (odd regular) \\
$\{1,2,3,4\} $&$ 4 $&$ 3 $&$ 68 $ & tight (odd regular) \\
$\{4\} $&$ 13/8 $&$ 27 $&$ 50890 $ & unknown\\
\end{tabular}
\end{center}
\caption{The terminating executions of our program, each one proving \eqref{eq:again} for that $\alpha$ and $T$. \emph{Tight} means convex point sets are known with $\sum_{t \in T}m_t = \alpha n - O(1)$ and \emph{abstractly tight} means some periodic configuration has $\sum_{t \in T}m_t = \alpha n$ but we could not realize it convexly in the plane.}
\label{table:1}
\end{table}

\section{Abstract Tightness}\label{sec:tightness}
Our computer program can also generate tight examples. In Figure~\ref{fig:m3} we show two periodic configurations with $m_3 = \frac{3}{2}n$ with periods of 6 and 8 levels, respectively. (No other example has period less than 14.) We were not able to embed these examples as convex point sets in the plane, and at the same time we did not disprove that they were embeddable. Based on our attempts, it seems like there is no simple periodic embedding respecting the natural symmetries of the distance configurations. A disproof of realizability could be used in the program to get stronger results. For $m_2 + m_3 = \frac{9}{4}n$ we also have an abstractly tight periodic example which we could not realize (Fig.~\ref{fig:m23}).
\begin{figure}[htb]
\begin{center}
\includegraphics[scale=0.85]{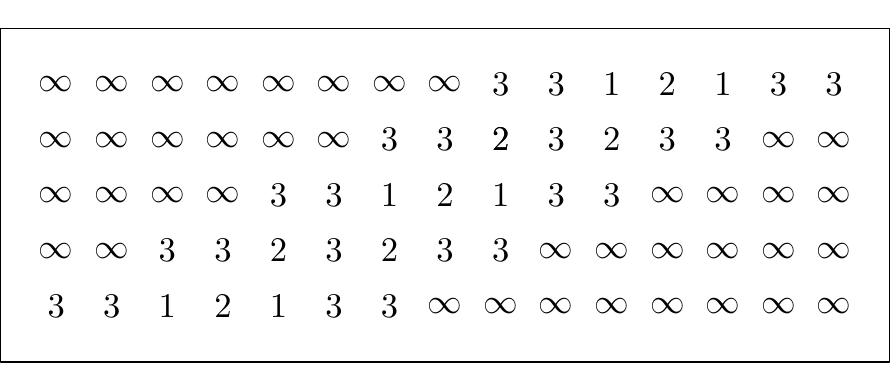}
\includegraphics[scale=0.85]{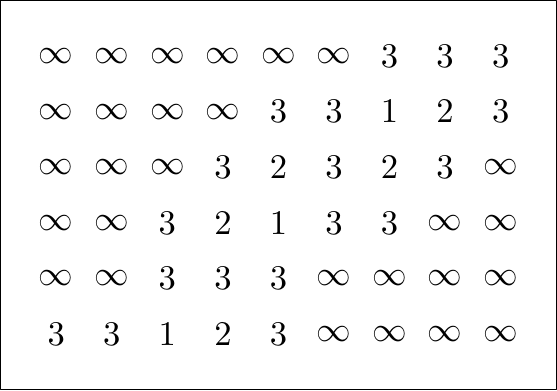}
\end{center}
\caption{Two unrealized periodic configurations with $m_3 = \frac{3}{2}n$. Rows and columns are two intervals of vertices, and entry $i$ (resp.~$\infty$) means distance $d_i$ (resp.~$<d_3$).}\label{fig:m3}
\end{figure}
\begin{figure}[htb]
\begin{center}
\includegraphics[scale=0.85]{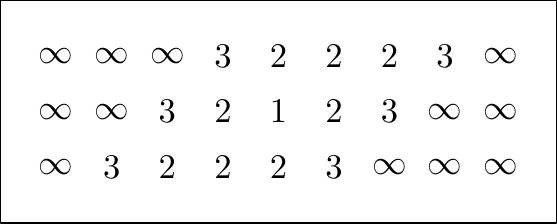}
\end{center}
\caption{An unrealized periodic configurations with $m_2 + m_3 = \frac{9}{4}n$.}\label{fig:m23}
\end{figure}

\section{Future Directions}\label{sec:future}
Our program is essentially a depth-first search; each configuration examined by the program has a unique ``parent" configuration from which it was grown. Thus, it would be possible to rewrite the program so as to use a smaller amount of memory and thereby possibly obtain results with smaller $\alpha$ or larger $k$; and a distributed implementation should also be straightforward.

It would be good to come up with constructions exhibiting better lower bounds. For example, no construction is known where $m_3/n$ is asymptotically greater than 4/3.

Our approach constitutes an abstract generalization of the original problem of bounding sums of the $m_i$'s in convex point sets. Vesztergombi~\cite{ve3} considered an abstraction as well, using only a subset of the facts we applied here. Can \prettyref{conjecture:unitdist} of Erd\H{o}s and Moser be violated in either of these abstractions?

Finally, can the functions $C, C'$ in \prettyref{lemma:constant-new} and \prettyref{lemma:constant-old} be improved?

\noindent
{\bf Acknowldegments.} We thank the referees for useful feedback, and K.~Vesztergombi for helpful discussions.



\typeout{Label(s) may have changed. Rerun}
\end{document}